\documentclass[12pt,a4paper,reqno]{amsart}
\usepackage[utf8]{inputenc}
\usepackage{enumerate}
\usepackage[margin=2.5cm]{geometry}
\usepackage{graphicx} 
\usepackage{overpic}
\usepackage{caption}

\usepackage{hyperref}

\newtheorem{lema}{Lemma}[section]
\newtheorem{tho}[lema]{Theorem}
\newtheorem{pro}[lema]{Proposition}

\newtheorem{coro}[lema]{Corollary}

\DeclareMathOperator{\e}{e}
\DeclareMathOperator{\tr}{tr}

\title[Resource-consumer dynamics in drylands]{Resource-consumer dynamics in drylands: modeling the role of plant-plant facilitation-competition shifts with a piecewise system}

\author[L.P.C. da Cruz]{Leonardo Pereira Costa da Cruz}
\address{Universidade de São Paulo, S\~ao Carlos, 13566--590 S\~ao Paulo, Brazil}
\email{leonardocruz@icmc.usp.br}

\author[J. Torregrosa]{Joan Torregrosa}
\address{Departament de Matem\`{a}tiques, Universitat Aut\`{o}noma de Barcelona, 08193 Be\-lla\-ter\-ra, Barcelona (Spain); Centre de Recerca Matem\`{a}tica, Campus de Be\-lla\-ter\-ra, 08193 Bellaterra, Barcelona (Spain)}
\email{joan.torregrosa@uab.cat}

\author[M. Berdugo]{Miguel Berdugo}
\address{Departamento de Biodiversidad, Ecolog\'{\i}a y Evoluci\'on, Universidad Complutense de Madrid, 28040 Madrid, Spain; Institute of Integrative Biology, Department of Environment Systems Science, ETH Zürich, Z\"urich, Switzerland}

\author[J. Sardany\'{e}s]{Josep Sardany\'{e}s$^*$}
\address{Centre de Recerca Matem\`{a}tica, Campus de Be\-lla\-ter\-ra, 08193 Bellaterra, Barcelona (Spain)}
\email{jsardanyes@crm.cat} 

\makeatletter
\@namedef{subjclassname@2020}{%
	\textup{2020} Mathematics Subject Classification}
\makeatother
\thanks{$^*$ Corresponding author.}



\begin{document}

\begin{abstract} 
In drylands, water availability determines plant population densities and whether they cooperate via facilitation or compete. When water scarcity intensifies, plant densities decrease and competition for water surpasses the benefits of soil improvement by facilitator plants, involving an abrupt shift from facilitation to competition. Here, we model this facilitation-competition shift using a piecewise system in a resource species such as grasses studying its impact on a resource-consumer dynamical system. First, the dynamics of each system are introduced separately. The competitive system, by setting conditions to have a monodromic equilibrium in the first quadrant, has no limit cycles. With a monodromy condition in the same quadrant, the cooperative system only has a hyperbolic, small amplitude limit cycle, allowing for an oscillating coexistence. The dynamic properties of the piecewise system become richer. We here prove the extension of the center-focus problem in this particular case, and from a weak focus of order three, we find 3 limit cycles arising from it. We also study the case assuming continuity in the piecewise system. Finally, we present a special and restricted way of obtaining a limit cycle of small amplitude in a pseudo-Hopf bifurcation type. Our results suggest that abrupt density-dependent functional shifts, such as those described in drylands, could introduce novel dynamical phenomena. Our work also provides a theoretical framework to model and investigate sharp density-dependent processes in Ecology. 
\end{abstract}

\maketitle

\section{Introduction}
Plant-plant interactions are one of the core mechanisms shaping the assemblage of a given community in ecosystems, importantly determining the identity and abundance of each species in a given place~\cite{Mittelbach2015}. Such interactions can be negative when plants compete for the same resources, but also positive, a process called facilitation \cite{deBruno2003}. Facilitation is especially important in stressful environments \cite{BertnessCallaway1994}, such as drylands (sites where it rains less than 65\% of what is evaporated~\cite{Cherlet2018}), where plants experience a chronic water deficit. In these systems facilitation emerges because plants, by shading and increasing soil organic content increase soil moisture in their surroundings \cite{FilazzolaLortie2014,Maestre2003}, creating micro-environmental conditions that promote the recruitment and growth of other species~\cite{CortinaMaestre2005}. However, recent studies have found that facilitation does not increase when the environment gets drier within drylands: it lessens its importance to drive species occurrence as aridity increases \cite{Berdugo2019,Zhang2022}. This occurs due to reasons that are still not clear \cite{Soliveres2015} but probably involve: (i) increasing aridity affects the quantity (system gets less productive) and the quality (as the soil is also less fertile with increasing aridity) of their litter, thus of the soil organic matter that ultimately improves microenvironmental conditions \cite{Berdugo2022}; (ii) increasing difficulty in producing an effective soil amelioration for recruitment due to harsher climatic conditions \cite{Zhang2022}; (iii) shifts in the plant species in the community as aridity increases \cite{Berdugo2019,Soliveres2015}, emerging species strategies to cope with water stress by developing deeper roots as they specialize to more arid conditions to access sub-soil water \cite{Berdugo2022}. 

The waning of facilitation as aridity increases is paralleled by an increase in the importance of competition between species owing to an increasing water scarcity, which ultimately tip the balance between facilitation and competition yielding systems that are fully governed by competitive interactions~\cite{Berdugo2019}. Importantly, such a shift does not occur smoothly as aridity increases but rather emerges abruptly at given specific aridity thresholds. Such abruptness is manifested in facilitation by the emergence of different community assemblage drivers~\cite{Berdugo2019}, an abrupt waning of soil amelioration \cite{Berdugo2020}, and by a change of the spatial patterns of vegetation (which ultimately emerge due to plant-plant interactions \cite{Berdugo2019b}). Moreover, the abrupt nature of the facilitation-to-competition shift is also documented to affect different components of ecosystems including soil microbial communities, soil fertility, shifts in vegetation dominant types (more dominated by shrubs), abrupt changes in the soil textural properties (which modulate water availability for plants) and drastic reduction of the sensitivity of vegetation to seasonal droughts \cite{Berdugo2020,Zhang2023}. All these changes probably indicate an abrupt restructuring of an ecosystem, involving the emergence of new rules attaining their structure, functioning, and dynamics. 

Concerning dynamics, the modeling and investigation of ecological functional shifts is scarce in the literature~\cite{Bassols2021,Perona2020}. In past years, a big interest in piecewise differential systems has emerged, because many real phenomena can be modeled with this class of systems e.g., electrical and mechanical systems, in control theory, and genetic networks~\cite{AcaBonBro2011,BerBudCha2008,Fil1988}. Usually, the simplest models are defined via planar piecewise polynomial vector fields $Z=(Z_1,Z_2)$ in the following way. Taking $0$ as a regular value of the function $h:\mathbb{R}^2\rightarrow \mathbb{R}$, we denote the separation curve by $\Sigma=h^{-1}(0)$ and the two regions it delimits by $\Sigma_i=\{ (-1)^i h(x, y)>0 \}$. So, the piecewise vector field can be written as 
\begin{equation}\label{eq:1}
	Z_i:(\dot{x},\dot{y})=(X_i(x,y),Y_i(x,y)),
	\text{ for } (x,y)\in \Sigma_i,
\end{equation}
where $X_i$ and $Y_i$ are polynomials of degree $n$ in $\Sigma_i,$ with $i=1,2$.
The above piecewise vector field is continuous when it satisfies $Z_1=Z_2$ on the separation curve $\Sigma.$ Otherwise we will say that it is discontinuous. 
The local trajectories of $Z$ on $\Sigma$ were stated by Filippov in \cite{Fil1988}, see Fig.~\ref{fi:filipov}.

The points on $\Sigma$ where both vectors fields simultaneously point outward or inward from $\Sigma$ define the \emph{escaping} ($\Sigma^e$) and \emph{sliding region} ($\Sigma^s$), respectively. The interior of its complement on $\Sigma$ defines the \emph{crossing region} ($\Sigma^c$), and the boundary of these regions is constituted by tangential points of $Z_i,$ with $\Sigma.$ 
\begin{figure}
\begin{center}
\begin{overpic}{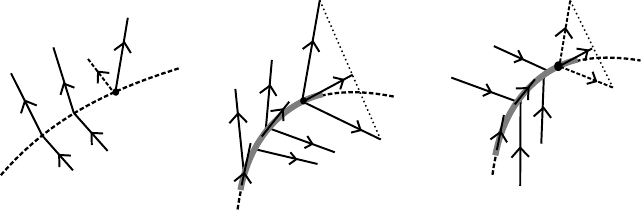}
\put(29,20){$\Sigma$}
\put(62.5,16){$\Sigma$}
\put(101,21){$\Sigma$}
\put(0,10){$\Sigma_1$}
\put(32,10){$\Sigma_1$}
\put(72,10){$\Sigma_1$}
\put(2,2){$\Sigma_2$}
\put(40,0){$\Sigma_2$}
\put(83,2){$\Sigma_2$}
\put(16,14){$p$}
\put(46,14){$p$}
\put(86,18){$p$}
\put(9,30){$Z_1(p)=Z_i(p)$}
\put(38,30){$Z_1(p)$}
\put(94,14){$Z_1(p)$}
\put(9,24){$Z_2(p)$}
\put(57,7){$Z_2(p)$}
\put(77,30){$Z_2(p)$}
\put(56,21){$Z_i(p)$}
\put(93,24){$Z_i(p)$}
\end{overpic}
\end{center}
 \captionsetup{width=\linewidth}
\caption{Definition of the vector field on $\Sigma$ following Filippov's convention in the sewing, escaping, and sliding regions, with $i=1,2.$}\label{fi:filipov}
\end{figure}
Let $Z_i h,$ denote the derivative of the function $h$ in the direction of the vector $Z_i$ that is, $Z_i h(p)=\langle \nabla h(p), Z_i(p)\rangle$. 
Notice that $p\in\Sigma^c$ provided that $Z_1 h(p)\cdot Z_2 h(p) > 0,$ $p\in\Sigma^e\cup\Sigma^s$ provided that $Z_1 h(p)\cdot Z_2 h(p) < 0,$ and $p$ in $\Sigma$ is a tangential point of $Z_i,$ provided that $Z_1 h(p)Z_2 h(p)=0.$ We say that $p \in \Sigma$ is a \emph{pseudo-equilibrium} of $Z,$ if $p$ is either a tangential point or an equilibrium of $Z_1$ or $Z_2.$ We call $p\in \Sigma$ an \emph{invisible fold} of $Z_1$ (resp. $Z_2$) if $p$ is a tangential point of $Z_1$ (resp. $Z_2$) and $(Z_1)^2 h(p)<0$ (resp. $(Z_2)^2 h(p)>0$).
A point on the separation curve $\Sigma $ is called multivalued if it has more than one distinct vector field defined. Otherwise, we will say that the point on $\Sigma$ is univalued e.g., all points are univalued in continuous piecewise vector fields. In general, the convention given by Filippov was defined to make sense of the lack of uniqueness of solution in a piecewise system (see Fig.~\ref{fi:filipov}). 

Let us consider that both differential equations in \eqref{eq:1} are Kolmogorov systems \cite{Kolmogorov1936}. Then, a planar piecewise polynomial Kolmogorov differential system is a planar dynamical system of the form
\begin{equation}\label{eq:2}
 Z_i=\begin{cases}
 \dot{x}=xX_{n-1,i}(x,y),\\
 \dot{y}=yY_{n-1,i}(x,y),
 \end{cases} \text{if} \ \ (x,y)\in \Sigma_i, 
\end{equation}
\\
where $X_{n-1,i}$ and $Y_{n-1,i},$ with $i=1,2,$ are polynomials of degree $n-1$. Particularly when $n=2$, we have the piecewise Lotka--Volterra systems. This class of systems has a wide range of applications, including chemical reactions \cite{Her1990}, economics \cite{Gandolfo2008,Goodwin1967,SoRic2002} and hydrodynamics \cite{Bus1981}. In this article, we provide a model for a resource-consumer system taking into account an abrupt ecological shift between dominant facilitation to full competition in the resource species i.e., plants. As usual and to simplify notation and computations, the equilibrium point $(x^*,y^*),$ being $x^*, y^*\in\mathbb{R}^+$, is located in the first quadrant, where the Kolmogorov systems have biological meaning. Moreover, we consider the case where it is located on the separation curve $\Sigma$. By a simple rescaling, $(x,y)\rightarrow (x^* x, y^* y),$ we can easily prove, if necessary, that it is not restrictive to assume that, in fact, it can be located at $(1,1)$. 

The model for competition is given by
\begin{equation}\label{eq:3}
	Z_1=\begin{cases}
		\dot{x}=x (k_1 (1-n_1 x)-e_1 y-w_1),\\
		\dot{y}= y (e_1 p_1 x-s_1 y-h_1),
	\end{cases}
\end{equation}
while the model including facilitation reads:
\begin{equation}\label{eq:4}
	Z_2=\begin{cases}
		\dot{x}=x (k_2 x (1-n_2 x)-e_2 y-w_2),\\
		\dot{y}= y (e_2 p_2 x-s_2 y-h_2).
	\end{cases}
\end{equation}
We propose a piecewise differential system that changes between competitive and cooperative dynamics using the piecewise differential system 
\begin{equation}\label{eq:5}
	Z= \begin{cases}Z_1 & \text{if} \ \ (x,y)\in \Sigma_1=\{0\leq x<1\},\\
		Z_2 & \text{if} \ \ (x,y)\in \Sigma_2=\{x>1\} ,\\
	\end{cases}
\end{equation}
where the separation line is $\Sigma=\{(x,y): x=1\}.$ On it we follow, as usual, the Filippov convention (for further details see~\cite{Fil1988}). Consequently, in the left hand side of the vertical straight line $\Sigma$ we propose a quadratic differential system considering only competition in the resource species given by Eqs.~\eqref{eq:3} while in the right hand side a cubic differential system with dominance of facilitation is taken into account [Eqs.~\eqref{eq:4}]. The system with facilitation is modelled as an autocatalytic process with a growth term of the form $k_2\,x^2$ which results in hyperbolic growth dynamics instead of an exponential one~\cite{Sardanyes2010,JosVidBlaiErnest2021}. The parameters for the resource population $x$ are given by the intrinsic growth rates $k_j>0$, being $j=1,2$; intra-specific competition $n_j> 0$; consumption rate $e_j>0$; and natural mortality $w_j\geq 0$. The case $w_j = 0$ considers that the main source of mortality is due to consumption. Concerning the consumer species, $y$, parameters are consumption rates $e_j>0$; and $0 < p_j <1$ denotes the fraction of energy invested in reproduction due to the consumption of the resource. Constants $s_j$ will be explored considering two different ecological processes for the consumer species: (i) $s_j > 0$: intra-specific competition; (ii) $s_j < 0$: intra-specific cooperation. Finally, $h_j>0$ are natural death rates for the consumer. If $\alpha_j = e_j p_j$, then $0<\alpha_j < e_j$ to fulfil the condition $0 < p_j < 1.$ To better differentiate between the competition and facilitation dynamics of each subsystem we will assume $n_1>n_2$. By doing so, we limit the competition term of the system with facilitation, which may also undergo some competition but having facilitation as a dominant process. Here we are not explicitly considering the availability of water in the model affecting the population density and the switch from facilitation to competition. Instead, we are using the separation line $\Sigma$, which is dependent on the resource population densities (assumed to be modulated by water availability): above a given density of the resource species $x^*=1$, facilitation dominates. At lower densities, due to a lack of water, interactions for the resource species become purely competitive. That is, our framework provides a phenomenological description of a density-dependent abrupt shift between facilitation and competition.

Our main results, stated in Theorem~\ref{mainteo} below, are the monodromic-type equilibria and the oscillatory motions around it. Consequently, we present a study of the bifurcations of \emph{limit cycles of crossing type} (a periodic orbit that cuts the separation curve in the set $\Sigma^c$). Hence, the crossing limit cycles are those that contain points in both regions. 

\begin{tho}\label{mainteo}There exists a piecewise differential system of the form \eqref{eq:5}, defined in two zones separated by a straight line, having three limit cycles.
\end{tho}

The paper is structured as follows. In Section~\ref{se:separately}, we discuss the facilitation and competition systems independently. Besides that, we show the necessary definitions and algorithms to obtain the coefficients of the return map, the so-called Lyapunov quantities for a smooth system. In Section~\ref{se:necsuficondition}, we investigate the piecewise system, also showing the algorithm to get Lyapunov quantities for a non-smooth system, presenting a new result to guarantee the existence of a crossing limit cycle in a Hopf-type bifurcation when we restrict our family to be a Kolmogorov family. This bifurcation is a generalization of the pseudo-Hopf, which is the phenomenon of the bifurcation of crossing limit cycles by adding constant terms in the piecewise system and consequently giving rise to a set of sliding of stability contrary to the already existing stability of pseudo-equilibrium. In general and usual unfoldings, this bifurcation breaks the Kolmogorov structure of system~\eqref{eq:2}. Here, we will prove that there exists a specific way to get this extra limit cycle inside the piecewise Kolmogorov class. We will see that the limit cycles of Theorem~\ref{mainteo} are of small amplitude and are found using a degenerated Hopf-bifurcation that provides only this maximal number. Consequently, this lower bound is, in fact, a lower bound of such kind of limit cycles.

\section{Qualitative dynamics of the models with competition and facilitation}\label{se:separately}

For the sake of clarity and completeness, in this section, we provide a summary of the dynamics of each of the two dynamical systems separately. In this work, we are interested in the study primarily in the neighborhood of an equilibrium point of monodromic type, so in this section, we will prove two lemmas which give conditions on the parameters for systems \eqref{eq:3}-\eqref{eq:4} to have a monodromic equilibrium point. To provide information about the phase portrait of these systems we take particular conditions to the parameters. We use the theory introduced below in Subsection~\ref{lyapunovconstants} to get the stability of a monodromic equilibrium point to prove Theorem~\ref{center1}, which proves the center problem to the model with facilitation, and finally in Proposition~\ref{fc}, we study the bifurcation of limit cycles.

\subsection{Algorithms to get the coefficients of the return map for an analytic system}\label{lyapunovconstants} 
	
Here, we recall how to obtain the coefficients of the return map near an equilibrium point of monodromic type, the so-called Lyapunov quantities. So, the equilibrium point will be a weak focus or a center. Let us consider a polynomial differential system of degree $n$ having an equilibrium point at the origin such that the eigenvalues of their Jacobian matrix are purely imaginary. Under this assumption, the matrix trace for the linear part of the system, at the equilibrium point, is zero. For simplicity, we will consider only the cases when the linear part of each system is written in its normal form as follows
\begin{equation}\label{eq:6}
 \begin{cases}
 \dot{x}=\tau x -y+\sum\limits_{k=2}^{\infty} X_k(x,y),\\
 \dot{y}=x+\tau y+\sum\limits_{k=2}^{\infty} Y_k(x,y),
 \end{cases}
\end{equation}
where $X_k,$ $Y_k$ denote homogeneous polynomials of degree $k$ for $ k\geq 2.$ Thus, doing the usual change to polar coordinates, $(x,y)=(r\cos\theta,r\sin\theta),$ and writing as a power series in $r$, we can write it as
\begin{equation}\label{eq:7}
 \frac{\mathrm{d} r}{\mathrm{d} \theta}=\sum_{k=1}^{\infty}S_k(\theta)r^k, 
\end{equation} 
where $S_1(\theta)=\tau$ and $S_k(\theta)$ are trigonometric polynomials for $k\geq2.$ For every $0<r_0\ll 1,$ we denote by $r(\theta,r_0)$ the solution of \eqref{eq:7} such that $r=r_0$ when $\theta=0$ and so
\begin{equation*}
 r(\theta,r_0)=\e^{\theta \tau}r_0+\sum_{k=2}^{\infty}u_k(\theta)r_0^k.
\end{equation*} 
Then, substituting the above solution in \eqref{eq:7}, we obtain a sequence of recurrent formulas to get the coefficients $u_k.$ Hence, directly
from the Poincar\'e return map $\Pi(r_0)=r(r_0,2\pi),$ we can define the displacement map as 
\begin{equation}\label{eq:8}
\Delta(r_0)=\Pi(r_0)-r_0=\sum_{k=1}^{\infty}V_kr_0^k,
\end{equation}
as it is illustrated in Fig.~\ref{fi:retunmap_0}.
\begin{figure}
	\begin{overpic}{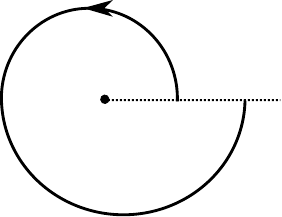}
		\put(30,30){$(0,0)$}
		\put(60,33){$r_0$}
		\put(75,45){$\Pi(r_0)$}
	\end{overpic}
	 \captionsetup{width=\linewidth}
\caption{The Poincar\'e return map $\Pi\left(r_0\right).$ }\label{fi:retunmap_0}
\end{figure}
In this context and when $\tau=0$, it is well-known that the first non-vanishing coefficient of \eqref{eq:8} has an odd subscript and $V_{2k+1}$ is called the $k$th-order Lyapunov quantity of \eqref{eq:6} and each one is defined only when the previous vanish. Then we say that the origin is a weak focus of order $k.$ When we consider these quantities as functions of the coefficients of $X_k$ and $Y_k$ in \eqref{eq:6}, we can prove that they are polynomials (see for example \cite{CimGasManMan1997}). An interesting property of these coefficients, described in \cite{Rou1998} and proved in \cite{CimGasMan2020}, is that for each $k$, we have 
\begin{equation*}
\langle V_2,V_4,\ldots, V_{2k}\rangle \subset \langle V_3,V_5,\ldots, V_{2k+1}\rangle,
\end{equation*}
with $k=1,\ldots$ When $V_3\ne0$, the stability of the equilibrium point is given by the sign of $V_3.$ More concretely, the origin is stable (resp. unstable) when $V_3<0$ (resp. $V_3>0$). Consequently, as $V_1=\e^{2\pi\tau}-1$, a stable (resp. unstable) limit cycle of small amplitude bifurcates from the origin taking the trace (equivalently $\tau$) a small enough positive (resp. negative) real number. This bifurcation is known as the classical Hopf bifurcation. The degenerate Hopf bifurcation is the natural generalization when the limit cycles of small amplitude appear, similarly, from a weak focus of order $k$. For a proof that from a weak focus of order $k$ bifurcate at most $k$ limit cycles see \cite{Rou1998}. This precise relation motivates the notion of order of a weak focus in the following.

\subsection{Resource-consumer model with no facilitation}
For completeness, we here provide a summary of well-known properties of the system with competition in the resource species, given by Eqs.~\eqref{eq:3}. As we will see below, this system has no limit cycles. Before proving it, we illustrate the qualitative dynamics by means of an inspection of the equilibrium points and the nullclines, together with some results on linear stability analysis. This system has four equilibrium points including one with a negative consumer population, given by $P_c^*=(0,- h_1/s_1)$. The biologically-meaningful equilibrium points are the origin $P^*_0 = (0, 0)$, $P_r^*=(x_r^*, 0)$ with $$x_r^* = \frac{1}{n_1} - \frac{w_1}{k_1 n_1},$$ and the interior point $P_{rc}^* = (x^*_{rc},y_{rc}^*)$ involving coexistence provided stable, with
$$
x_{rc}^* = \frac{s_1(k_1 - w_1) + h_1 e_1}{e_1^2 p_1 + s_1 k_1 n_1}
, \qquad {\rm and} \qquad y_{rc}^* = \frac{k_1 (1 - n_1 x_{rc}^*) - w_1}{e_1}.$$
It is easy to show that the eigenvalues at the origin are $\lambda_1(P_0^*) = k_1 - w_1$ and $\lambda_2(P_0^*) = - k_1$, meaning that this point will be locally asymptotically stable when $w_1 > k_1$ and of saddle type when $w_1 < k_1$. Moreover, the eigenvalues of the equilibrium with resource's persistence and consumer's extinction are $\lambda_1(P_r^*) = w_1 - k_1$ and 
$$
\lambda_2(P_r^*) = e_1 p_1 \left(\frac{1}{n_1} - \frac{w_1}{k_1 n_1}\right) - h_1
.$$ 
The $x$-nullcline is given by the $y$-axis and 
\begin{equation*}
x=\frac{1}{n_1} - \frac{e_1 y - w_1}{k_1 n_1},
\end{equation*}
while the $y$-nullcline is given by the $x$-axis and 
\begin{equation*}
y=\frac{e_1 p_1 x - h_1}{s_1}.
\end{equation*}

\begin{figure}
\includegraphics[width=0.8\textwidth]{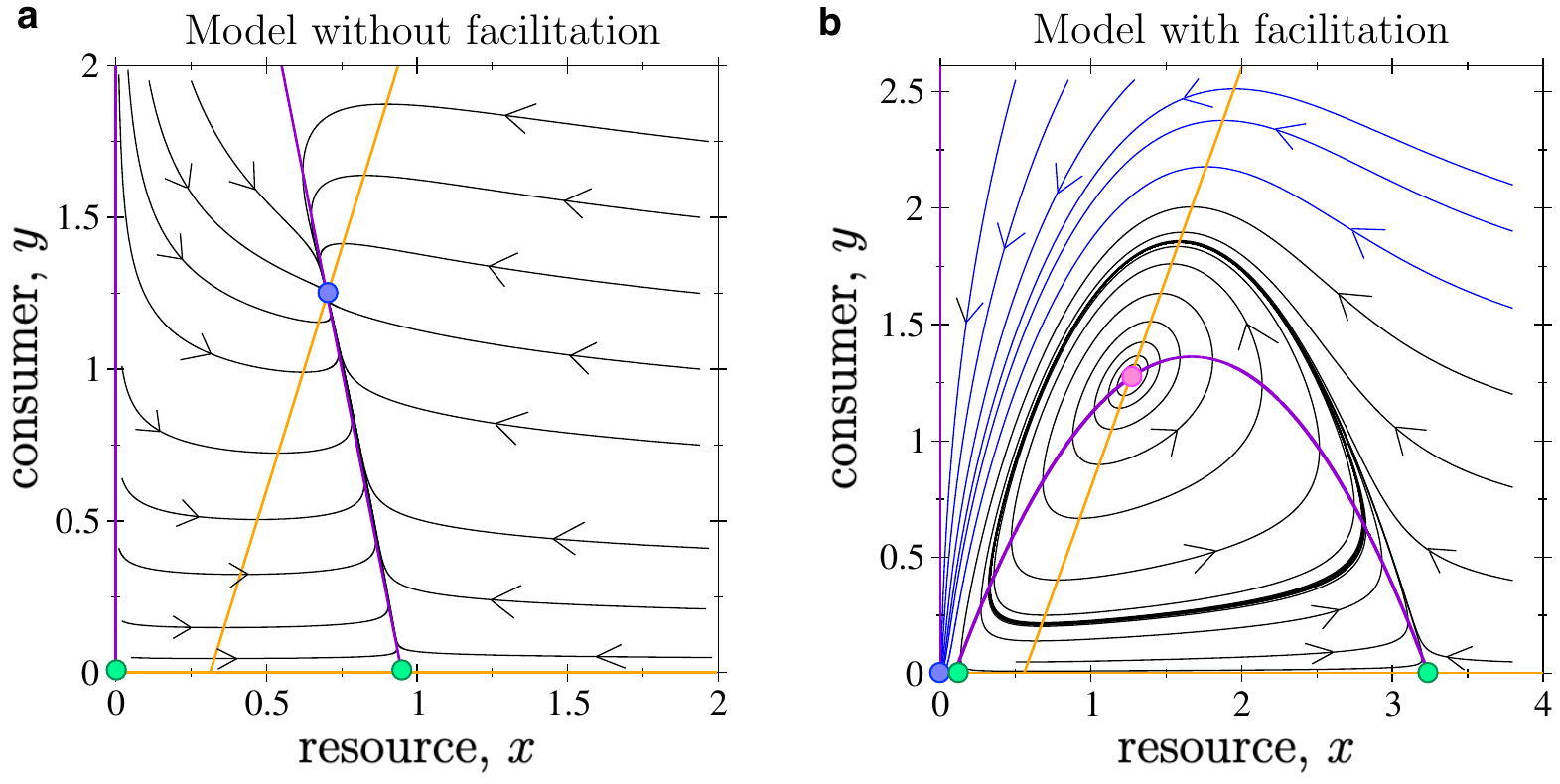}
 \captionsetup{width=\linewidth}
 \caption{\small{(a) Monostability and coexistence for the purely competitive system \eqref{eq:3} governed by a stable equilibrium, with $k_1 = n_1 = 1$, $e_1 = 0.2$, $p_1 = 0.8$, $s_1 = w_1 = h_1 = 0.05$. (b) Bistability between resource-consumer extinction (blue orbits) and coexistence governed by a limit cycle in the system with facilitation given by Eqs.~\eqref{eq:4}, with $k_1 = 0.55$, $n_1 = e_1 = 0.3$, $w_1 = 0.05$, $p_1 = 0.6$, and $s_1 = h_1 = 0.1$. For each system, we also include the $x$-nullcline (violet) and the $y$-nullcline (orange). Blue and green circles are attractors and saddle points, respectively. The pink circle is an unstable focus within the limit cycle. The arrows indicate the directions of the orbits.}}
\label{fig_4}
\end{figure}

Figure~\ref{fig_4}(a) shows a phase portrait with some interior orbits, the equilibrium previously found (colored circles) and the nullclines. For the selected parameter values, the dynamics unfolds coexistence and both fixed points $P_0^*$ and $P_r^*$ are saddle points and the equilibrium $P_{rc}^*$ is a stable node.

In the following result, we provide conditions on the parameters so that the equilibrium point $P_{rc}^*$ is the point $(1,1)$, and that it is of monodromic type. More concretely, it will be of weak focus type. In the proof we will see that, up to a time rescaling, we can fix the value of the determinant of the Jacobian matrix at $(1,1).$ 

\begin{lema}\label{l1} Assuming the conditions 
\begin{equation}\label{eq:9}
 \begin{aligned}
 h_1=&\frac{k_1^2 n_1^2+e_1 k_1 n_1+1}{e_1}, & p_1 &= \frac{k_1^2 n_1^2+1}{e_1^2},\\ 
 s_1=&-k_1 n_1, & w_1&=-k_1 n_1-e_1+k_1,
 \end{aligned} 
\end{equation}
the differential system 
\eqref{eq:3} has an equilibrium point of weak focus type at $(1,1)$. In fact, it is a center. 
 \end{lema}

\begin{proof}
The weak focus conditions at the equilibrium point at $(1,1)$ follows easily taking the system \eqref{eq:3} and solving the algebraic system 
\[
\{Z_1(1,1)=\tr Z_1(1,1)-t_1=\det Z_1(1,1)-a_1^2=0\},
\]
where $\tr$ and $\det$ are the trace and determinant, respectively, with $t_1, a_1 \in\mathbb{R}$. Directly, we obtain the following conditions
\begin{equation}\label{eq:10}
\begin{aligned}
h_1 =&(k_1^2 n_1^2 + e_1 k_1 n_1 + k_1 n_1 t_1 + a_1^2 + e_1 t_1)/e_1, & s_1 =& -k_1 n_1 - t_1,\\
p_1 =& (k_1^2 n_1^2 + k_1 n_1 t_1 + a_1^2)/e_1^2,&w_1 =& -k_1 n_1 - e_1 + k_1. 
\end{aligned}
\end{equation}
Hence, when $t_1=0$ and $a_1=1$ (note that changing time and rescaling parameters if necessary, this last condition $a_1=1$ is not restrictive), we get \eqref{eq:9}, then $(1,1)$ is an equilibrium point of weak focus type. Also note that, above the statement condition, system \eqref{eq:3} has the first integral $A(x,y) x^B y^C$ and the integrating factor $x^Dy^E,$
where 
\begin{equation*}
\begin{aligned}
A(x,y) = & \; n_1 k_1 (k_1^2 n_1^2 + e_1 k_1 n_1 + 1) x + e_1 k_1 n_1 (k_1 n_1 + e_1) y \\
& - (k_1 n_1 + e_1) (k_1^2 n_1^2 + e_1 k_1 n_1 + 1),\\
B = & \; (n_1 k_1 (k_1^2 n_1^2 + e_1 k_1 n_1 + 1)/e_1, \\
C = & \;k_1 n_1 (k_1 n_1 + e_1),\\
D =& \;(k_1^3\,n_1^3+e_1\,k_1^2\,n_1^2+k_1\,n_1-e_1)/e_1,\\
E =& \;k_1^2\,n_1^2+e_1\,k_1\,n_1-1.
\end{aligned}
\end{equation*}

Therefore, we conclude that the equilibrium point is a center. 
\end{proof}
As a consequence of the previous result, system \eqref{eq:3} does not have limit cycles. 

\begin{lema}
Assuming the conditions \eqref{eq:10}, and 
\begin{equation*}
\begin{aligned}
e_1 =&\frac{k_1^2 n_1^2 + k_1 n_1 t_1 + 1}{k_1 n_1 y_1 - k_1 n_1 + t_1 (y_1-1)}, \ \ k_1 =\frac{t_1 (x_1 +y_1-x_1 y_1)}{2 (x_1 y_1 - x_1 - y_1) n_1} \\
 &+\frac{(t_1^2 (x_1^2 y_1^2- 2 x_1^2 y_1 - 2 x_1 y_1^2 + x_1^2 + 2 x_1 y_1 + y_1^2)+ 4 (x_1 y_1-x_1 -y_1))^{1/2}}{2 (x_1 y_1 - x_1 - y_1) n_1},
\end{aligned}
\end{equation*}
with $x_1, y_1 \in \mathbb{R}.$ Then, the equilibria of system \eqref{eq:3} are
\begin{equation}\label{eq:11}
P_0^*=(0,0),\ \ P^*_{rc}=(1,1),\ \ P^*_c=(0,y_1),\ \ \text{and}\ \ P^*_r=(x_1,0).
\end{equation}
\end{lema}

We will not do the complete study here. However, we can easily study the behavior of the equilibrium points depending on the position in which each one is. For some examples, see Fig.~\ref{fig_5}. 

\subsection{Resource-consumer model with facilitation}

We here provide a summary of the system with facilitation, given in~\eqref{eq:4}. This system has been recently studied including a fraction of habitat loss or destroyed, $D \in [0,1]$, in the logistic growth function i.e.,  $k_2\,x^2\,(1 - D - n_2 x)$, with $n_2 = 1$ and also setting $s_2=0$. This particular system revealed how habitat loss can cause tipping points impacting on species' extinctions. This model also exhibited self-sustained oscillations and various local and global bifurcations, with associated ghost transients and critical slowing down (see~\cite{JosVidBlaiErnest2021} for details). 

As we did for the system without facilitation we here provide a summary of the equilibrium points and some insights into their local stability. This model has six equilibria, including the one which is not biologically-meaningful given by $\Omega_c^*=(0,- h_2/s_2)$ plus the origin $\Omega_0^*=(0,0)$ with eigenvalues $\lambda_1(\Omega_0^*) = -w_2$ and $\lambda_2(\Omega_0^*) = -h_2$, being asymptotically locally stable. Two more equilibria are given by $\Omega^*_{r^\pm} = (x^*_{r^\pm}, 0)$ with
$$
x^*_{r^\pm} = \frac{1}{2 n_2} \left(1 \pm \frac{(k_2^2 - 4k_2 n_2 w_2)^{1/2}}{k_2}\right)
.$$
From the previous expression we can derive the bifurcation value $k_2^{(c)} = 4 n_2 w_2$ at which the equilibria $x^*_{r^+}$ and $x^*_{r^-}$ collide in a saddle-node bifurcation.
The remaining equilibria, which is cumbersome, involved in the coexistence of the two species is obtained with the computation of the nullclines, we call $\Omega^*_{rc^\pm}=(x^*_{rc^\pm},y^*_{rc^\pm})$. The $x$-nullcline is given by the $y$-axis and by $x^*_{rc^-} \cup x^*_{rc^+}$, with $$x^*_{rc^\pm} = \frac{k_2 s_2-e_2^2 p_2\pm\Lambda^{1/2}}{2 k_2 n_2 s_2},$$ with $\Lambda = k_2^2 s_2^2-2(e_2^2 p_2-2 e_2 h_2 n_2+ 2 n_2 s_2 w_2)s_2k_2+e_2^4 p_2^2$, and $\Lambda \geq 0$.
The $y$-nullcline is given by the $x$-axis and 
\begin{equation*}
y^*_{rc^\pm}=\frac{e_2 p_2 x^*_{rc^\pm} - h_2}{s_2}.
\end{equation*}
As in Lemma~\ref{l1}, we will give conditions on the parameters of system \eqref{eq:4}, so that the equilibrium point $P_{rc^+}^*$ is of weak focus type and, as we have detailed in the introduction, it is located at $(1,1).$ For this purpose, consider the following result.

\begin{lema}
Assuming the conditions 
\begin{equation}\label{eq:12}
 \begin{aligned}
 h_2=&(4 k_2^2 n_2^2+2 e_2 k_2 n_2-4 k_2^2 n_2-e_2 k_2 +k_2^2+1)/e_2,& s_2&=-2 k_2 n_2+k_2,\\
 p_2=&(4 k_2^2 n_2^2-4 k_2^2 n_2+k_2^2+ 1)/e_2^2,& w_2&=-k_2 n_2-e_2+k_2,
 \end{aligned} 
\end{equation}
in the differential system \eqref{eq:4}. Then, $(1,1)$ is an equilibrium point of monodromic type. More concretely, it is of weak focus type.
 \end{lema}
 \begin{proof}
 Taking \eqref{eq:4}, and solving the algebraic system $$\{Z_2(1,1)=\tr Z_2(1,1)-t_2=\det Z_2(1,1)-a_2^2=0\},$$ 
where $t_2, a_2 \in\mathbb{R},$ and $\tr$ and $\det$ are the trace and determinant operators, respectively, we get the solution
\begin{equation*}
\begin{aligned}
h_2 =& (4 k_2^2 n_2^2 + 2 e_2 k_2 n_2 - 4 k_2^2 n_2 + 2 k_2 n_2 t_2 + a_2^2 - e_2 k_2 + e_2 t_2 + k_2^2 - k_2 t_2)/e_2,\\
p_2 =& (4 k_2^2 n_2^2 - 4 k_2^2 n_2 + 2 k_2 n_2 t_2 + a_2^2 + k_2^2 - k_2 t_2)/e_2^2,\\ 
s_2=&-2 k_2 n_2 + k_2 - t_2, \\ 
w_2 =& -k_2 n_2 - e_2 + k_2.
\end{aligned}
\end{equation*}
Here, when $t_2=0$ and $a_2=1$ (note that by scaling the system parameters, this last condition $a_1=1$ is not restrictive), we get the situation \eqref{eq:12}, then $(1,1)$ is a weak focus.
 \end{proof}

\begin{tho}\label{center1} For family \eqref{eq:4} satisfying \eqref{eq:12}, the equilibrium point $(1,1)$ is a center if and only if $\hat{V}_3=0,$
where
\begin{equation}\label{eq:13}
\hat{V}_3=k_2^3(8n_2^3-12n_2^2-6n_2-1)+k_2^2(8e_2n_2^2-6e_2n_2+e_2)+k_2(2n_2-1)+2e_2.
\end{equation}
\end{tho}
\begin{proof}

\begin{figure}
\vspace{1cm}´
\begin{overpic}[width=\textwidth]{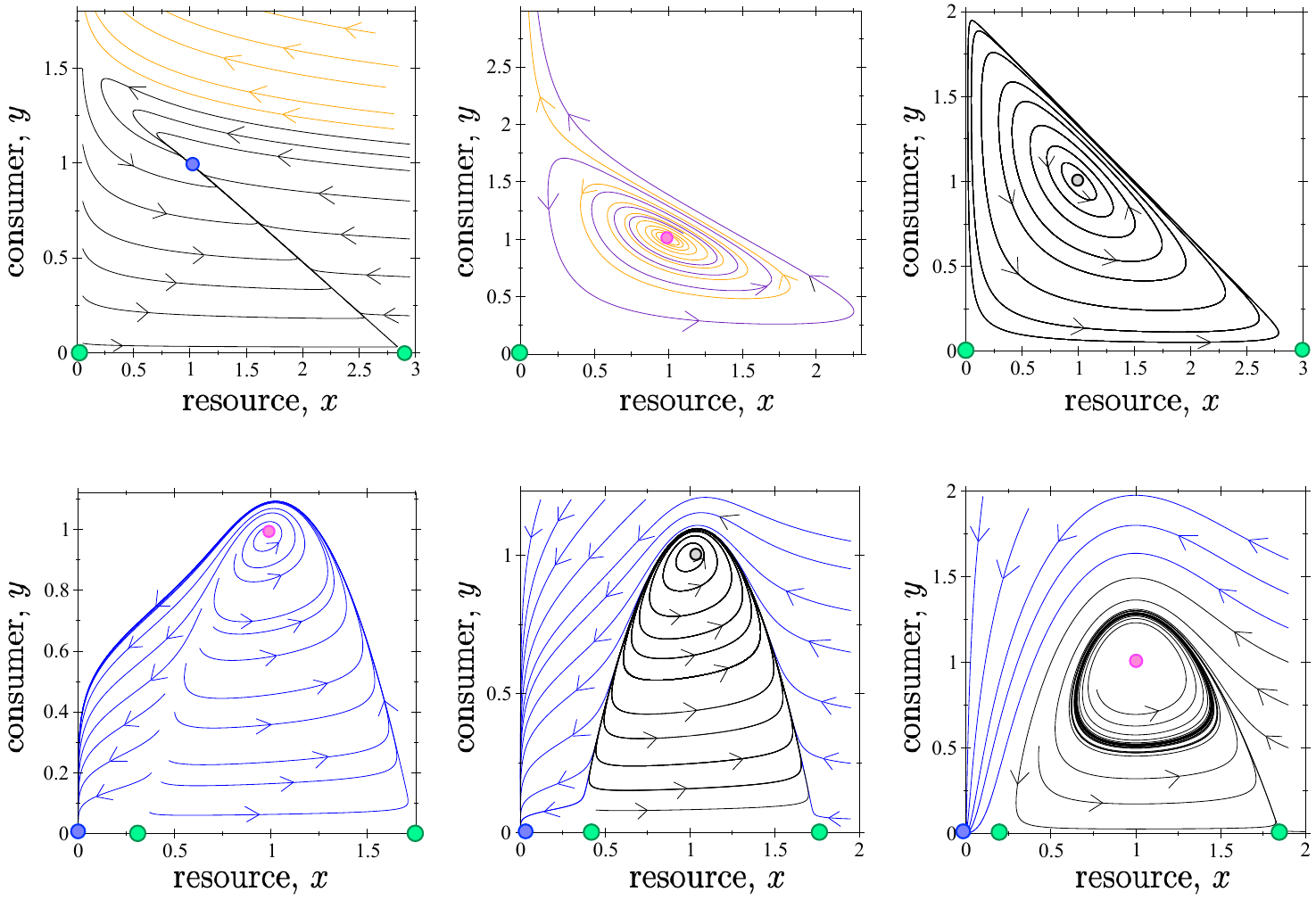}
	\put(3,70){Model without facilitation}
	\put(27.5,65){(a)}
	\put(61,65){(b)}
	\put(95,65){(c)}
	\put(3,33){Model with facilitation}	
	\put(27.5,28){(a)}
	\put(61,28){(b)}
	\put(95,28){(c)}
\end{overpic}	
 \captionsetup{width=\linewidth}
	\caption{\small{Examples of several phase portraits having a coexistence equilibrium at $(1, 1)$. 
	\\(Upper row, competitive model) Equilibria~\eqref{eq:11}: (a) stable node with $s_1<0$ and thus with cooperation in consumers' reproduction. The orange orbits involve an asymptotic extinction of the resource species and the dominance of consumers with $a_1=1/4, t_1=-1, x_0=29/10, y_0=16/10$; (b) extinction of the resource and dominance of the consumer via an unstable coexistence focus also with $s_1<0,$ and $a_1=1, t_1=1/10, x_0=3, y_0=2$; (c) coexistence via center now with $s_1>0,$ $a_1=1, t_1=0, x_0=3, y_0=2$. 
	\\ (Lower row, model with facilitation) (a) Unstable coexistence focus driving to co-extinction, which we obtain by considering a small perturbation of the center given in (b); (b) resource-consumer coexistence via center-driven oscillations with equilibria~\eqref{eq:15}, $x_0=1/3, x_1=7/4$; (c) coexistence via self-sustained oscillation governed by a limit cycle, $s_2=0, p_2=1, e_2=266/2025, h_2=266/2025, k_2=92/225, n_2=100/207, w_2=2/25$. The arrows indicate the direction of the orbits. Blue orbits in the model with facilitation start within the basin of attraction of the origin involving co-extinctions.}}
	\label{fig_5}
\end{figure}

The first necessary condition to have a nondegenerate equilibrium point of center-focus type at $(1,1)$ of \eqref{eq:4} is holding, assuming \eqref{eq:12}, i.e, the trace and the determinant of the Jacobian matrix are zero and one, respectively. So, doing the translation $x\rightarrow x+1$ and $y\rightarrow y+1$ to put the monodromic equilibrium point $(1,1),$ at the origin and performing the linear change
\begin{equation*}
x\rightarrow \frac{ (2 k_2 n_2 -k_2) x+e_2 y}{e_2}, \ \ y\rightarrow\frac{x}{e_2},
\end{equation*}
we get the following differential system
\begin{equation}\label{eq:14}
\begin{cases}
\begin{aligned}
\dot{x}=&\, y+k_2 (2 n_2 - 1) x^2 - (4 k_2^2 n_2^2 + 2 e_2 k_2 n_2 - 4 k_2^2 n_2 - e_2 k_2 + k_2^2 - 1)x y\\& - k_2 (2 n_2 - 1) (e_2 k_2 n_2 + 1) y^2- n_2 e_2^2 k_2^2 (2 n_2 - 1) y^3 ,\\
\dot{y}=&-x - e_2 x y- e_2 k_2 n_2 y^2 -e_2^2 k_2 n_2 y^3.
\end{aligned} 
\end{cases}
\end{equation}
 Straightforward computations using the algorithm given in Subsection~\ref{lyapunovconstants} allow us to get $V_1=0$ and 
\begin{equation*}
 V_3= \frac{\pi}{4} e_2 k_2 n_2 \hat{V}_3,
\end{equation*}
where $\hat{V}_3$ is given in \eqref{eq:13}. It is easy to check that $V_5=V_7=0$. As we are only studying under the assumptions $e_2,k_2,n_2>0,$ the necessary condition for giving a center is the one in the statement. The proof follows checking that we have a first integral. For shortness, we provide a Darboux first integral, $A(x,y) x^{B} y^{C}$, of initial system~\eqref{eq:4} and the corresponding integrating factor $x^Dy^E,$ being
\begin{equation*}
\begin{aligned}
A(x,y)=\: & 2((8 n_2^2 - 6 n_2 + 1)k_2^2 + 2)(n_2 x - 1)x + k_2^2 (2 n_2 - 1)^2 y + 2(2 n_2-1 )k_2^2 n_2+2,\\
B =&-2,\\ 
C =&-(2 n_2-1) k_2^2/(8 k_2^2 n_2^2-6 k_2^2 n_2+k_2^2+ 2),\\
D=&-3,\\ 
E=&-2(4k_2^2n_2^2-2k_2^2n_2+1)/(8k_2^2n_2^2-6k_2^2n_2+k_2^2+ 2).
\end{aligned}
\end{equation*}
We note that for some values of the parameters we will maybe need to do a change on it to guarantee that it will be well defined. This is clear from the fact that for \eqref{eq:14} the Taylor series of the first integral starts as $x^2+y^2+\cdots.$
\end{proof}
We notice that system \eqref{eq:4} has also centers when $n_2=0$ or $k_2=0$, despite such parameter values are not interesting from an ecological point of view.

\begin{pro}\label{fc} The maximal weak focus order of the equilibrium point $(1,1)$ of the differential system \eqref{eq:4} is one. This maximal property is obtained when the parameters satisfy the condition \eqref{eq:12} and are in $\mathcal{F}=\{\hat{V}_3\neq 0\},$ where $\hat{V}_3$ is given in \eqref{eq:13}. Additionally, this weak focus unfolds in \eqref{eq:4} at most one limit cycle of small amplitude.
\end{pro}
\begin{proof} From the proof of Theorem~\ref{center1} it is clear that the maximal weak focus order is one and is obtained when $t_2=0$ and $\mathcal{F},$ because $V_1=0$ and $V_3|_{\mathcal{F}}\neq0.$ The limit cycle emerges from the origin using the classical Hopf bifurcation being $t_2$ small enough and $t_2V_3|_{\mathcal{F}}<0.$
\end{proof}

The equilibria of this system are displayed in Fig.~\ref{fig_5}(c) for a scenario with coexistence governed by a limit cycle. Under the used parameter values, the origin is locally asymptotically stable, the equilibria placed at the axis $y=0$ are saddle points and the interior equilibrium point is an unstable focus surrounded by a limit cycle.

To give an idea of the phase portrait of the system \eqref{eq:4}, we make a summary by assuming \textit{a priori} that the system has a monodromic equilibrium point. Specifically, we chose the center under conditions given in Theorem~\ref{center1}. So, consider the next result. 

\begin{lema}
Assuming the center conditions given in Theorem~\ref{center1} and defining 
\begin{equation*}
 \begin{aligned}
 n_2=& \frac{1}{x_0 + x_1},\\
 k_2=& \left(\frac{2(n_2x_0-1)x_0+1}{
 	(8n_2^2-6n_2 +1)(n_2x_0-1)x_0+2n_2^2-n_2
 }\right)^{1/2},
 \end{aligned} 
\end{equation*}
with $x_0, x_1\in \mathbb{R}.$ Then, the equilibrium points of system \eqref{eq:4} are 
\begin{equation}\label{eq:15}
\begin{aligned}
\Omega_0^*=&(0,0), \ \ \Omega^*_{r^+}=(x_0,0),\ \ \Omega^*_{r^-}=(x_1,0), \ \ \Omega_c^*=\left(0,\frac{2 x_0 x_1}{2 x_0 x_1-x_0-x_1}\right), \\
\Omega^*_{rc^+}=&(1,1),\ \ \Omega^*_{rc^-}=\left(\frac{x_0 x_1 (x_0+x_1-2)}{2 x_0 x_1-x_0-x_1},-\frac{x_0 x_1(x_0- x_1)^2}{(2 x_0 x_1-x_0-x_1)^2}\right).
\end{aligned}
\end{equation}
\end{lema}
We will not do the complete study here. However, we can easily study the behavior of the equilibrium points depending on the position in which each one is. For some examples, see Fig.~\ref{fig_5}. 
	
\section{Modelling facilitation-competitions abrupt shifts: piecewise dynamics}\label{se:necsuficondition}

In this section, we prove the main result of this paper announced in Theorem~\ref{mainteo}. To do so we will introduce, in Subsection~\ref{lyapunovconstantspiecewise}, the basic results on the stability of a monodromic equilibrium point on $\Sigma$ (generalized Lyapunov quantities) to analyze the centers and the bifurcation of crossing limit cycles of small amplitude. We start restricting our study to nonexistence of sliding segments and we will finish considering them in the model. Throughout this section, we will assume \textit{a priori} conditions \eqref{eq:9} and \eqref{eq:12} 
i.e., the point $(1,1)$ is a monodromic equilibrium point located on $\Sigma$ of weak focus type.

\subsection{Algorithms to get the coefficients of the return map to a piecewise system}\label{lyapunovconstantspiecewise} 

Let us consider a planar piecewise analytic differential system written as
\begin{equation}\label{eq:16}
 Z_i=\begin{cases}
 \dot{x}=\tau_i x-y +\sum\limits_{k=2}^{\infty} X_{k,i}(x,y),\\
 \dot{y}=x+\tau_i y+\sum\limits_{k=2}^{\infty} Y_{k,i}(x,y),
 \end{cases} \text{if} \ \ (x,y)\in \Sigma_i, 
\end{equation}
with $\Sigma_i=\{(x,y):(-1)^{i+1} y>0\},$ and $X_{k,i},$ $Y_{k,i}$ being homogeneous polynomials of degree $k,$ for $i=1,2.$ As mentioned above, after the polar coordinates change, we have for $0<r_0\ll 1,$ the power series of the piecewise solution, which satisfies $r_1(0,r_0)=r_2(\pi,r_0)=r_0,$ reads as
\begin{equation*}
 r(\theta,r_0)=\begin{cases}
 r_1(\theta,r_0)=\e^{\tau_1 \theta}r_0+\sum\limits_{k=2}^{\infty}u_{k,1}(\theta)r_0^k, \text{ if } \theta\in [0,\pi], \\
 r_2(\theta,r_0)=\e^{\tau_2 \theta}r_0+\sum\limits_{k=2}^{\infty}u_{k,2}(\theta)r_0^k, \text{ if } \theta\in [\pi,2\pi].
 \end{cases} 
\end{equation*}
Therefore, we define the positive and negative Poincar\'e half-return maps as $\Pi_1(r_0)=r_1(\pi,r_0)$ and $\Pi_2(r_0)=r_2(2\pi,r_0).$ Finally, we define the piecewise Poincar\'e return map by the composition of the two half-return maps, $\Pi_2\left(\Pi_1(r_0)\right).$ To simplify computations, instead of considering the previously introduced displacement map, we will use the equivalent difference map 
\begin{equation}\label{eq:17}
\Delta(r_0)=\left(\Pi_2\right)^{-1}(r_0)-\Pi_1(r_0)=-\sum_{k=1}^{\infty}V_kr_0^k,
\end{equation}
which is illustrated in Fig.~\ref{fi:retunmap}.
\begin{figure}
	\begin{overpic}{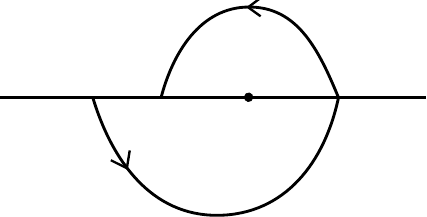}
		\put(55,20){$(0,0)$}
		\put(80,23){$r_0$}
		\put(35,20){$\Pi_1(r_0)$}
		\put(5,32){$\left(\Pi_2\right)^{-1}(r_0)$}
		\put(90,40){$Z_1$}
		\put(90,10){$Z_2$}
		\put(105,27){$\Sigma$}
	\end{overpic}
	 \captionsetup{width=\linewidth}
	\caption{The positive and negative half-return maps $\Pi_1$ and $\left(\Pi_2\right)^{-1},$ respectively. }\label{fi:retunmap}
\end{figure}
As we have introduced in the analytical case, the first non-vanishing coefficient in \eqref{eq:17}, $V_k\neq0,$ is called the (generalized) $k$th-order Lyapunov quantity of \eqref{eq:16}. We notice that system \eqref{eq:16} has no sliding segment at the origin. In fact, in \eqref{eq:17} it is clear that $\Delta(0)=0$. In this context, the origin of \eqref{eq:16}, which is on $\Sigma,$ will be a (crossing) weak focus of order $k,$ when $V_j=0,$ $1\leq j\leq k-1,$ and $V_k\neq0.$ The aim of such definition is that, in a complete unfolding, $k$ limit cycles of small amplitude bifurcate from the origin, similarly to the analytic degenerated Hopf bifurcation. Note that, as $V_1=\e^{\pi\tau_1}-\e^{-\pi\tau_2},$ the weak focus condition is $\tau_1+\tau_2=0.$ As it is more intricate to work with than the one in the analytic case, it is natural to restrict our higher order analysis to $\tau_1=\tau_2=0.$ As we will see at the end of the work, we will recover the condition for $\tau_1+\tau_2$ only in the pseudo-Hopf bifurcation phenomenon. That is, the smallest crossing limit cycle appearing in the unfolding.

\subsection{Maximal order of a weak focus and the bifurcation of crossing limit cycles of small amplitude}

In planar piecewise vector fields, the sufficient condition to get a center is more complicated than in smooth vector fields. For smooth systems, if the equilibrium point is of monodromic type and there is a first integral, it is sufficient to prove the condition of the equilibrium point being center. However, for non-smooth systems, we must prove that the positive and negative half-return maps satisfy $\Pi_1(r_0)-\left(\Pi_2\right)^{-1}(r_0)=0$, see Fig.~\ref{fi:retunmap}. Different situations about how to check this condition can be seen, for example, in \cite{CruzTorre2022}. In particular, getting the first integrals in each region is not enough to have a local, well defined first integral and continuous in an open set. We will introduce the notion of \emph{$\Sigma$-first integral} when we have first integrals in each region, $\Sigma_1$ and $\Sigma_2$ in our case. The usual definition of first integral, that is, a non-constant function that is constant along the solutions implies the continuity condition. Consequently, if we have a continuous piecewise first integral around a monodromic pseudo-equilibrium we will have a center. However, in general, a $\Sigma$-first integral will not be always a first integral but we can have a center, as we will see in the next result. Moreover, we want to emphasize that unlike what happens in smooth differential equations, a transformation in a piece differential equation can modify the separation curve, which would entail a change in the global behavior of the solutions. In particular, the one that refers to closed orbits, since they could break with a general change. This problem can be avoided using twin $\Sigma$-transformations, that are simultaneous change of variables in the two regions that coincide on the separation curve. In this way, it is guaranteed that periodic orbits are transformed into periodic orbits.

\begin{tho}\label{center2} For family \eqref{eq:5}, under the conditions
\eqref{eq:9} and \eqref{eq:12}, the equilibrium point $(1,1)$ is a center if and only if $\hat V_2=\hat V_3=0$. With
\begin{equation*}
	\hat V_2=4(e_2+k_2 n_2-k_2)k_2n_2-(k_1 n_1+e_1)k_1 n_1+(k_2-e_2)k_2
\end{equation*}
and $\hat V_3$ defined in \eqref{eq:13}.
\end{tho}

\begin{proof}
In order to simplify notation, although we will do some changes of coordinates, we do not change the names of the sets $\Sigma,$ $\Sigma_1$, and $\Sigma_2;$ nor the vector fields $Z_1$ and $Z_2.$	

Under the hypothesis of the statement both systems $Z_1$ and $Z_2$ have a weak focus point at $(1,1).$ Moreover, the corresponding linear matrices of the vector fields at the equilibrium have zero trace and the determinant is one. So, the necessary condition to have a nondegenerate equilibrium point of center-focus type at $(1,1)$ holds. The algorithm given in Subsection~\ref{lyapunovconstantspiecewise} for finding the Lyapunov quantities requires the translation $x\rightarrow x+1,$ and $y\rightarrow y+1$ to locate the equilibrium point at the origin. Note that this translation also moves the separation straight line to $x=0$.
 
Next, we apply the Jordan change to the system without facilitation defined after the translation in $\Sigma_1=\{(x,y):x<0\},$ given by
\begin{equation*}
x\rightarrow \frac{k_1 n_1 x+e_1 y}{e_1}, \ \ y\rightarrow\frac{x}{e_1},
\end{equation*}
and, to the system with facilitation, defined in $\Sigma_2=\{(x,y):x>0\},$ we apply the change 
$$x\rightarrow \frac{ (2 k_2 n_2 -k_2) x+e_2 y}{e_2}, \ \ y\rightarrow\frac{x}{e_2}.$$
Consequently, we obtain a piecewise system written in the form \eqref{eq:16} with $\tau_1=\tau_2=0$. The first quadratic differential system (purely competition) reads
\begin{equation*}
 Z_1=\begin{cases}
 \dot{x}=y+k_1 n_1 x^2-(k_1^2 n_1^2+e_1 k_1 n_1-1) xy -k_1 n_1 y^2,\\
 \dot{y}=-x-e_1 x y,
 \end{cases} \text{if} \ \ (x,y)\in \Sigma_1,
\end{equation*}
while the second cubic differential system (dominance of facilitation) is given by
\begin{equation*}
Z_2=\begin{cases}
\begin{aligned}
\dot{x}=& y+k_2 (2 n_2 - 1) x^2 + (-4 k_2^2 n_2^2 - 2 e_2 k_2 n_2+ 4 k_2^2 n_2\\
& + e_2 k_2 - k_2^2 + 1)x y - k_2 (2 n_2 - 1) (e_2 k_2 n_2 + 1) y^2\\
&-n_2 e_2^2 k_2^2 (2 n_2 - 1) y^3 ,\\
\dot{y}=&-x - e_2 x y- e_2 k_2 n_2 y^2-e_2^2 k_2 n_2 y^3,
\end{aligned} 
\end{cases} \text{if} \ \ (x,y)\in \Sigma_2,
\end{equation*}
where $\Sigma=\{(x,y): y=0\},$ $\Sigma_1 =\{(x,y): y<0\},$ and $\Sigma_2 =\{(x,y): y>0\}.$ Hence, as the system is in its normal form, we apply again the algorithm of Subsection~\ref{lyapunovconstantspiecewise}. Straightforward computations provide the first Lyapunov quantities $V_k,$ being polynomials with rational coefficients in $(e_i,k_i, n_i)$ for $i=1,2.$ We get 
\begin{equation}\label{eq:18}
\begin{aligned}
 V_2=&\frac{2}{3} \hat V_2,\,\, V_3=\frac{\pi}{8} e_2 k_2 n_2 \hat{V}_3, \text{ and } V_4=V_5=0.
\end{aligned}
\end{equation}
It is clear that the solutions of the algebraic system $\{V_2=V_3=0\}$ give us the necessary center conditions provided in the statement, because we are assuming the biological conditions $e_2,k_2,n_2>0.$

For the sufficient condition, we recover the original coordinates and start looking for a $\Sigma$-first integral of the form $H_i(x,y)=A_{i}(x,y) x^{B_{i}} y^{C_{i}}$ and a $\Sigma$-integrating factor $W_i(x,y)=x^{D_i}y^{E_i},$ for $i=1,2$ associated to the initial system~\eqref{eq:5}. The reader can check that the corresponding functions and exponents are 
\begin{equation*}
\begin{aligned}
A_{1}(x,y)=& 2 k_1^2 n_1^2 ((2 n_2-1)^2k_2^2 + 1) ( (8 n_2^2 - 6 n_2 + 1) k_2^2+2) x \\
&+ k_2^2 (2 n_2 - 1) ((8 n_2^2 - 6 n_2 +1)k_1^2k_2^2n_1^2 + 2 k_1^2 n_1^2 + (2 n_2 -1) k_2^2)
 y \\ &+ 2 k_2^2 (2 n_2 - 1) ((2 n_2-1)^2k_2^2 + 1),\\
B_{1} =& -2 k_1^2 n_1^2 ((2 n_2-1)^2k_2^2 + 1)/(((8 n_2^2 - 6 n_2 + 1) k_2^2 + 2) k_1^2 n_1^2 + (2 n_2 \!-1) k_2^2),\\ 
C_{1} =& -(2 n_2-1) k_2^2/((8 n_2^2 - 6 n_2 + 1) k_2^2+2),\\ 
D_1=&-\frac{(16 n_2^2 - 14 n_2 + 3)k_1^2 k_2^2 n_1^2 + 4 k_1^2 n_1^2 + (2 n_2 - 1)k_2^2}{(8 n_2^2 - 6 n_2 +1)k_1^2k_2^2n_1^2 + 2 k_1^2 n_1^2 + (2 n_2 -1) k_2^2},\\
E_{1}=&-2 (2(2 n_2 - 1)k_2^2n_2 + 1)/( (8 n_2^2 - 6 n_2 + 1) k_2^2+2),
\end{aligned}
\end{equation*}

\begin{equation*}
\begin{aligned}
A_{2}(x,y)=& ((8 n_2^2 - 6 n_2 + 1)k_2^2 + 2)(n_2 x - 1 )x+(2 n_2 - 1)^2 k_2^2 y/2 + (2 n_2 - 1)k_2^2n_2 + 1,\\
B_{2} =&-2,\\ 
C_{2} =& (1-2 n_2) k_2^2/(2 + (8 n_2^2 - 6 n_2 + 1) k_2^2),\\
D_{2}=&-3,\\ 
E_{2}=&-2 (2(2 n_2 - 1)k_2^2n_2 + 1)/((8 n_2^2 - 6 n_2 + 1)k_2^2 + 2).
\end{aligned}
\end{equation*}
Because $(1,1)$ is a weak focus, the solutions cut the separation line at $(1,u)$ and $(1,v),$ satisfying $H_i(1,u)=H_i(1,v),$ for $i=1,2$, and $u\approx 1,$ then, each half return map is defined as $v=\Pi_i(u).$ For this family, it can be seen that $H_i(1,u)=\gamma_i \hat H(u),$ with
\begin{equation*}
	\begin{aligned}
 \hat H(u)=&(1+(4 n_2+u-2) (2n_2-1) k_2^2/2) \, u^{\frac{(1-2 n_2) k_2^2}{(8 n_2^2-6 n_2+1) k_2^2+2}} ,\\
\gamma_1=&(2 n_1^2 (4 n_2-1) (2 n_2-1) k_1^2+4 n_2-2) k_2^2+4 k_1^2 n_1^2,\\
\gamma_2=&2n_2-1.
\end{aligned}
\end{equation*}
Consequently, the functions $v=\Pi_i(u)$ satisfy $H_i(1,u)-H_i(1,v)=\gamma_i\hat H(u)-\gamma_i\hat H(v)=0,$ for $i=1,2,$ that is the same condition as $\hat H(u)-\hat H(v)=0.$ So, both half return maps coincide and the proof follows.
\end{proof}

We remark that if, as in the previous section, we do not consider ecologically relevant parameters,~\eqref{eq:5} has other families of centers: $\{k_1=k_2=0\};$ 
 $\{n_1=k_2=0\};$ $\{n_2= e_2 k_2+k_1^2 n_1^2+e_1 k_1 n_1 -k_2^2=0\};$ and
$\{k_2=e_1+k_1 n_1=0\}.$

\medskip

The next result is an immediate consequence of the proof above.

\begin{coro}\label{ccff3}
The maximal order of a weak focus in family \eqref{eq:5} is three. Moreover, it is located at $(1,1).$ 
\end{coro}
	
\begin{pro}\label{ff3} There are at most two limit cycles of small amplitude bifurcating from $(1,1)$ under the conditions of monodromy \eqref{eq:9} and \eqref{eq:12} when we consider family \eqref{eq:5}, restricted to $\mathcal{S}=\{k_i(1-n_i)-e_i-w_i=e_ip_i-s_i-h_i=0, i=1,2\}.$
\end{pro}

\begin{proof}
We will prove that the two limit cycles emerge from the pseudo-equilibrium $(1,1)$ via a degenerated Hopf bifurcation. 

From Section~\ref{lyapunovconstantspiecewise} we know that $V_1=\e^{\pi\tau_1}-\e^{-\pi\tau_2}$ and in the proof of Theorem~\ref{center2} we have taken $\tau_1=\tau_2=0.$ So $V_1=0.$ Moreover, condition $\mathcal{S}$ forces family \eqref{eq:5} to have no sliding nor escaping segments. Consequently, $(1,1)$ is a weak focus or a center. In the first case, we know that there exist values of the parameters such that the Lyapunov quantities associated to it are $V_1=V_2=0$ and $V_3\ne0$. It is easy to see that only one limit cycle of small amplitude can bifurcate from $(1,1)$ when $V_1=0,$ taking $V_2$ small enough and $V_2V_3<0$. When $(1,1)$ is a center, that is when $\{V_1=V_2=V_3=0\}$ (see again the proof of Theorem~\ref{center2} if necessary), again only one limit cycle of small amplitude can bifurcate from $(1,1)$ when $V_1=0.$ In this latter case, the upper bound follows from Theorem~9 of Chapter~2 in \cite{Rou1998}, because the ideal $I=\left<V_2,V_3\right>$ is radical. 

In both cases the independence of the parameters appearing in the definition of $V_2$ in \eqref{eq:18} with respect to $\tau_1$ and $\tau_2$ guarantee that we can unfold only one more limit cycle of small amplitude near $(1,1),$ because we have no sliding nor escaping segments.
\end{proof}

In the proof above, it is clear that under the restriction of the nonexistence of sliding or escaping segments, the difference map \eqref{eq:17} vanishes at the equilibrium point and this property is maintained under perturbation. Consequently, if the first nonvanishing coefficient of the return map is the corresponding to power three, only two limit cycles of small amplitude can bifurcate. As we will see in the next result, in this family the continuity condition does not decrease the number of limit cycles of small amplitude. This notion is why we say that, for a continuous piecewise vector field, the weak focus order is one less than the subscript of the first non-vanishing coefficient of the return map. Because the number of limit cycles should be related to the weak focus order. In this sense, the notion of order for a weak focus depends on the family of vector fields we are analyzing. 

The conditions in the following proposition come from the continuity of \eqref{eq:5} on $\Sigma$ and the existence of an equilibrium point located at $(1,1)$. Its proof follows repeating point by point all the steps of the proof of Proposition~\ref{ff3} checking all the conditions under the restriction given in the statement.
\begin{pro} 
The conclusion of Proposition~\ref{ff3} remains true under the conditions 
	\begin{equation*}
		\begin{aligned}
\mathcal{C}=\{&e_1 = e_2= (1-n_2)k_2-w_2,
w_1 = (1-n_1)k_1-e_2,\\
&h_1 = ((1-n_2)k_2 +1)w_2(p_2 -p_1)+h_2,s_1 =s_2= e_2p_2-h_2
\}.
 \end{aligned}
	\end{equation*}
\end{pro}

In (discontinuous) piecewise differential systems, one more limit cycle can be obtained from the monodromic-type equilibria since we can change the stability of the equilibrium adding a sliding or escaping segment. When the pseudo-equilibrium is a monodromic point of fold-fold quadratic type this phenomenon was denominated as pseudo-Hopf bifurcation in \cite{KuzRinGra2003}, but proved previously in \cite{Fil1988}. A collection of similar Hopf-type bifurcations can be found in \cite{Sim2022}. We are now interested in monodromic pseudo-equilibria that are in fact equilibria for both systems. We can not use the generic unfolding of this Hopf-type bifurcation found in \cite{CruNovTor2019,FrePonTor2014} because the Kolmogorov structure is broken. The following result allows to overcome this obstacle.

\medskip

We will denote by the pair $[X_1,X_2]_\Sigma$ the piecewise system defined by
\begin{equation*}
	\begin{cases}X_1 & \text{if} \ \ (x,y)\in \Sigma_1=\{0\leq x<1\},\\
		X_2 & \text{if} \ \ (x,y)\in \Sigma_2=\{x>1\} ,\\
	\end{cases}
\end{equation*}
where the separation line is $\Sigma=\{x=1\}$.
\begin{pro}\label{hopftype} Consider the piecewise Kolmogorov system $[Z_1,Z_2]_\Sigma$ defined by $Z_i:=(\dot{x},\dot{y})=(x f_i(x,y),y g_i(x,y))$ for $i=1,2,$ having at $(1,1)$ an unstable monodromic-type equilibrium which rotates counter-clockwise. Then, the partial perturbed piecewise system $[Z_{1,\varepsilon},Z_2]_\Sigma$ (resp. $[Z_{1},Z_{2,\varepsilon}]_\Sigma$) exhibits an unstable limit cycle of small amplitude in a Hopf-like bifurcation around $(1,1)$ for $\varepsilon>0$ (resp. $\varepsilon<0$) small enough. The perturbed system $Z_{i,\varepsilon}$ is $Z_i$ under the homothetic change $(x,y)\rightarrow ((1+\varepsilon) x,(1+\varepsilon) y).$ 
\end{pro}	
\begin{proof}
Let us consider only the case of $\varepsilon>0$. The proof is a direct consequence of the Poincaré--Bendixson Theorem for piecewise differential systems~\cite{BuzCarEuz2018}. More concretely, taking $\varepsilon$ small enough, the equilibrium $(1,1)$ of $Z_{1,\varepsilon}$ moves from the separation line $\Sigma$ to $\Sigma_2$ being a virtual (or invisible) equilibrium of $Z_1.$ Then, an attracting sliding segment over $\Sigma$ appears, changing the stability of the neighborhood of the point $(1,1)$ and an unstable limit cycle of small amplitude bifurcates. See Fig.~\ref{fi:hopftype} for details.

\end{proof}

\begin{figure}
	\begin{overpic}[height=7.3cm]{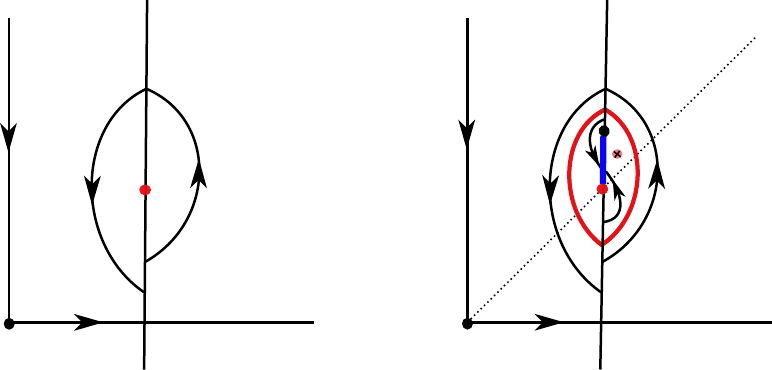}
	\end{overpic}
	 \captionsetup{width=\linewidth}
	\caption{The Hopf-type bifurcation. (Left) Unstable monodromic-type equilibria. (b) Limit cycle together with the sliding segment after applying the homothetic transformation in the convenient direction. We have depicted in blue (red) the asymptotic stable (unstable) objects.}\label{fi:hopftype}
\end{figure}

We notice that the result above guarantees the existence of a limit cycle for the specific signal choose of the perturbation parameter $\varepsilon.$ Nothing is said in the converse direction, because other bifurcations can occur and they are not the objective of this work. 

\medskip

The following two results provide points in the parameter space where the weak foci of
Corollary~\ref{ccff3} are respectively unstable and stable and both have biological meaning. In both, Proposition~\ref{ff3} applies and two limit cycles appear from a degenerated Hopf bifurcation without sliding or escaping segments. But Proposition~\ref{hopftype} only can be used in the first one to obtain a third limit cycle. Consequently, our main result Theorem~\ref{mainteo} is proved.

\begin{pro}
The differential system \eqref{eq:5} has an unstable weak focus of order three at $(1,1)$ when the parameters are 
\begin{equation}\label{eq:19}
\begin{aligned}
\mathcal{T}^u=&\Bigg\{n_1=\frac{1}{4}; e_1=2; k_1=\frac{\sqrt{401}-1}{5}; h_1=\frac{19\sqrt{401}+381}{400}; p_1=\frac{401-\sqrt{401}}{800}; \\
& s_1=\frac{1-\sqrt{401}}{20}; w_1=\frac{3\sqrt{401}-43}{20}; n_2=\frac{1}{10}; e_2=\frac{619-19\sqrt{401}}{300};k_2=\frac{5}{2}; \\& h_2=\frac{285\sqrt{401}+4517}{2384}; p_2=\frac{450000}{( 19\sqrt{401}-619)^2}; s_2=2; w_2=\frac{14}{75}-\frac{19\sqrt{401}}{300}\Bigg\}.
\end{aligned}
\end{equation}
Additionally, in piecewise system \eqref{eq:5}, there exist values of the parameters such that nearby $\mathcal{T}^u$ three limit cycles of small amplitude bifurcate from $(1,1).$ See Fig.~\ref{fi:bifurcation}.
\end{pro}
\begin{proof} The first part of the statement follows straightforward, computing the first Lyapunov quantities which are $V_1=V_2=0$ and $V_3=19 \pi (319 \sqrt{401}-4119)/288000>0.$ The second part is a direct consequence of the application of Propositions~\ref{ff3} and \ref{hopftype} in a consecutive way.
\end{proof}
\begin{figure}
	\begin{overpic}[width=\linewidth]{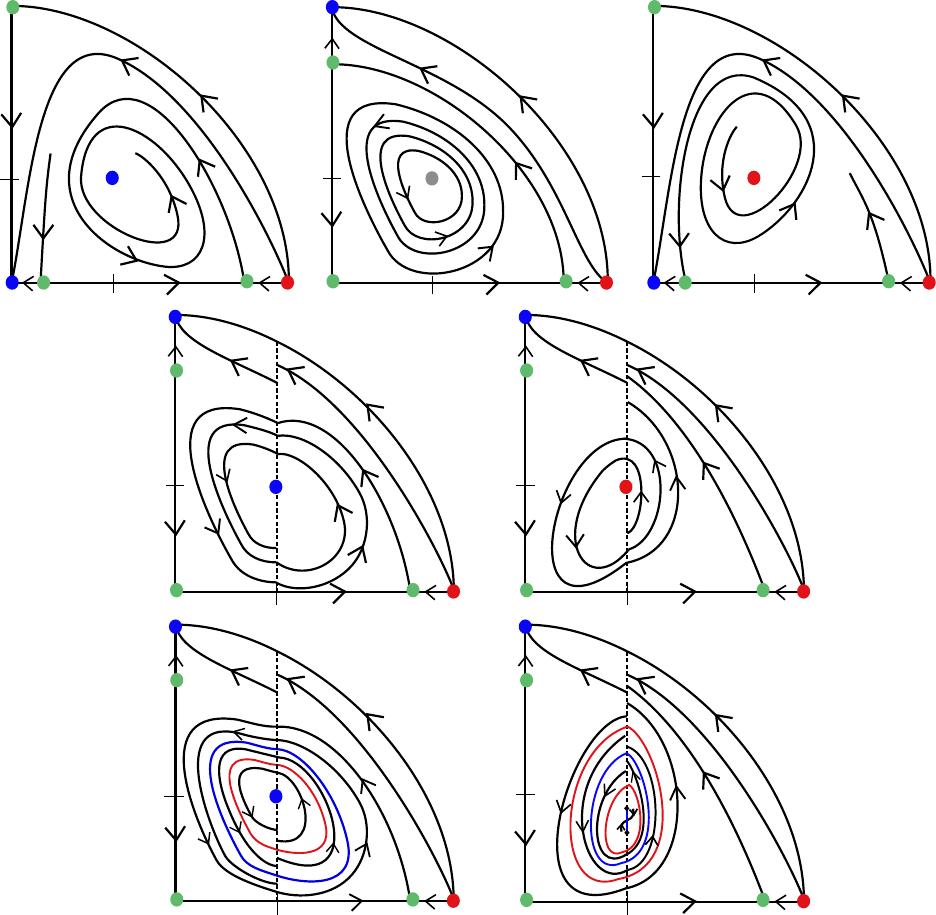}		
	\put(-5,95){(a)}		
	\put(29,95){(b)} 		
	\put(64,95){(c)} 		
	\put(12,62){(d)}
	\put(50,62){(e)}
	\put(12,29){(f)}
	\put(50,29){(g)}
\end{overpic}
 \captionsetup{width=\linewidth}
\caption{
Some qualitative phase portraits for the resource ($x$-axis) and consumer ($y$-axis) system drawn in the first quadrant of the Poincaré disk. We use blue (red) to represent asymptotic stable (unstable), green for saddles, and gray for centers. (b) Center for the model with no facilitation. For the system with facilitation we show in (a) and (c), stable and unstable weak focus, respectively. (d) Stable weak focus $\mathcal{T}^s$ given in \eqref{eq:20}, which is formed in piecewise configuration in the left the system (b), and in the right the system (a). The bifurcation of two limit cycles from the piecewise weak focus $\mathcal{T}^s,$ is shown in (f). (e) Unstable weak focus $\mathcal{T}^u,$ which is formed in piecewise configuration in the left (b) and in the right (c), given in \eqref{eq:19}. Finally, (g) displays the bifurcation of three limit cycles from the weak focus $\mathcal{T}^u,$ including the sliding segment.}\label{fi:bifurcation}
\end{figure}
\vspace{2cm}
\begin{pro}
	The differential system \eqref{eq:5} has a stable weak focus of order three at $(1,1)$ when the parameters are 
	\begin{equation}\label{eq:20}
		\begin{aligned}
			\mathcal{T}^s=&\left\{n_1=\frac{1}{4}; e_1=2; k_1=\frac{\sqrt{401}-1}{5}; h_1=\frac{19\sqrt{401}+381}{400}; p_1=\frac{401-\sqrt{401}}{800}; \right. \\
			& s_1=\frac{1-\sqrt{401}}{20}; w_1=\frac{3\sqrt{401}-43}{20}; n_2=\frac{501}{1000}; e_2=\frac{36199}{100400}+\frac{19\sqrt{401}}{502};\\ 
			& k_2=\frac{5}{2}; h_2=\frac{381529538\sqrt{401}}{4480072399}-\frac{722414019499}{896014479800}; p_2=\frac{10080412004}{(36199 + 3800\sqrt{401})^2}; \\
			& \left. s_2=-\frac{1}{200}; w_2=\frac{1781}{2008}-\frac{19\sqrt{401}}{502}\right\}.
		\end{aligned}
	\end{equation}
	Additionally, in the piecewise system \eqref{eq:5}, there exist values of the parameters such that nearby $\mathcal{T}^u$ two limit cycles of small amplitude bifurcate from $(1,1).$ See Fig.~\ref{fi:bifurcation}.
	
\end{pro}
\begin{proof}The first part of the statement follows straightforward, computing the first Lyapunov quantities which are $V_1=V_2=0$ and 
	\[
	V_3=-\frac{501 \pi (1430890268969 + 55558510831\sqrt{401})}{3225651200000000}<0.
	\]
	The second part is a direct consequence of the application of Proposition~\ref{ff3}.	
\end{proof}

\section{Conclusions}
Plant-plant interactions shape ecosystems, affecting species identity and abundance. Such interactions can be either positive (facilitation) or negative (competition), especially in drylands where water scarcity is common. Facilitation arises when plants improve soil moisture and conditions for growth~\cite{deBruno2003}. However, as dryness increases, facilitation weakens due to factors like declining soil quality, harsher climates affecting plant strategies, and increased competition for water~\cite{Berdugo2019,Zhang2023}. Hence, a lack of water involves a decrease in plant populations which strongly compete instead of cooperating. This shift from facilitation to competition occurs abruptly at specific dryness thresholds~\cite{Berdugo2020,Berdugo2019,Berdugo2019b}, causing significant changes in ecosystems, including vegetation patterns, soil properties, and reduced sensitivity to droughts. This abrupt change signals a restructuring of ecosystems, probably involving new rules governing their structure and dynamics. 

In this contribution, we have introduced a piecewise dynamical system to model abrupt ecological shifts involving changes from facilitation to competition driven by changes in the population densities of a resource species e.g., grasses. The availability of water has not been modeled explicitly but is indirectly considered with the increase in population densities establishing a density threshold above which plants establish facilitation. We have studied how this shift impacts a resource-consumer system considering a Holling type I functional response. We have first provided a summary of the dynamics for the two systems separately (see also~\cite{JosVidBlaiErnest2021} and references therein). The purely competitive model has an interior coexistence equilibrium point and no limit cycles are found. The model with facilitation allows for resource-consumer self-sustained oscillations through a limit cycle. 

These two systems have been coupled using a piecewise system considering an abrupt transition from facilitation to competition as the resource species decreases in population (due to the abiotic factor of water depletion). As mentioned, such abrupt thresholds have been described in field data for dryland ecosystems~\cite{Berdugo2019,Berdugo2019b}. Transitioning to a piecewise system reveals richer dynamics, demonstrating three limit cycles and an extended center-focus problem. Additionally, continuity in the piecewise system and a Hopf-type bifurcation are studied. Our findings reveal that abrupt ecological shifts in drylands can lead to new dynamic phenomena. For instance, an increase in the likelihood to have different deterministic, self-sustained oscillatory regimes. Our research also introduces a modeling framework to investigate abrupt, density-dependent functional shifts in population dynamics in Ecology. Further research should consider more complex functional responses for the consumer species such as Holling type II and III functions. Finally, we would like to highlight that, despite the difficulty of obtaining high-frequently sampled time series and due to environmental fluctuations, signals of piecewise dynamics could be searched in ecological time series for plants in drylands. 

\section{Acknowledgements}
This work has been funded by the Catalonia AGAUR agency (2021SGR00113 grant) (JT); the Spanish Ministerio de Ciencia, Innovaci\'on y Universidades - Agencia Estatal de Investigaci\'on (PID2022-136613NB-I00 grant) funded by MCIN/AEI/10.13039/501100011033 ‘ERDF A way of making Europe’ (JT); the Ram\'{o}n y Cajal grants RYC-2017-22243 (JS) and RYC-2021-031797-I (MB) funded by MCIN/AEI/10.13039/501100011033 ‘FSE invests in your future’; the Severo Ochoa and Mar\'{\i}a de Maeztu Program for Centers and Units of Excellence in R\&D (CEX2020-001084-M grant) (JT and JS); the Brazilian S\~ao Paulo Research Foundation FAPESP (2021/14987-8 and 2022/14484-9 grants) (LPC); and the European Community H2020-MSCA-RISE-2017-777911 grant (JT). We thank the CERCA Programme/Generalitat de Catalunya for institutional support.

\bibliographystyle{abbrv}
\bibliography{biblio.bib}
\end{document}